\documentclass[12pt,twoside]{amsart}
\usepackage{graphicx}
\usepackage{enumerate}
\usepackage{amsmath, amsthm, amsfonts}

\newtheorem{theorem}{Theorem}[section]

\newtheorem{proposition}[theorem]{Proposition}

\newtheorem{definition}[theorem]{Definition}

\begin{document}
\title{\bf Separation axioms of $\alpha^{m}$-open sets}
\date{}

\author{R. Parimelazhagan}

\address{Department of Mathematics, RVS Technical Campus, Coimbatore - 641 402, Tamilnadu, India. E-mail: pari$_{_{-}}$kce@yahoo.com.}

\author{Milby Mathew}

\address{Research Scholar, Department of Mathematics, Karpagam University, Coimbatore - 32, Tamilnadu, India. E-mail: milbymanish@yahoo.com}

\maketitle
\begin{abstract}
In this paper, we introduce $T_{\alpha^{m}}$-Spaces, $\alpha^{m}$-closed maps and $\alpha^{m}$-open maps and studied some of their properties.
\end{abstract}


\section{Introduction}

Separation axioms are one among the most common, important and interesting concepts in Topology. They can be used to define more restricted classes of topological spaces. The separation axioms of topological spaces are usually denoted with the letter $``T"$ after the German $``Trennung"$ which means separation. Most of the weak separation axioms are defined in terms of generalized closed sets and their definitions are deceptively simple. However, the structure and the properties of those spaces are not always that easy to comprehend. The separation axioms that were studied together in this way were the axioms for Hausdorff spaces, regular spaces and normal spaces. \\

Separation axioms and closed sets in topological spaces have been very useful in the study of certain objects in digital topology (\cite{khal1970}, \cite{kong1991}). Khalimsky, Kopperman and Meyer \cite{khal1990} proved that the digital line is a typical example of $T_\frac{1}{2}$-spaces. There were many definitions offered, some of which assumed to be separation axioms before the current general definition of a topological space. For example, the definition given by Felix Hausdorff in 1914 is equivalent to the modern definition plus the Hausdorff separation axiom. The first step of generalized closed sets was done by Levine in 1970 [5] and he initiated the notion of $T_\frac{1}{2}$-spaces in unital topology which is properly placed between $T_{0}$-space and $T_{1}$-spaces by defining $T_\frac{1}{2}$-space in which every generalized closed set is closed.  \\

The concept of generalized closed sets(briefly $g$-closed) in topological spaces was introduced by Levine \cite{le1970} and a class of topological spaces called $T_\frac{1}{2}$ spaces. Arya and Nour \cite{ar1990}, Bhattacharya and Lahiri \cite{bh1987}, Levine \cite{le1963}, Mashhour \cite{ma1982}, Njastad \cite{nj1965} and Andrijevic(\cite{an1986}, \cite{an1996}) introduced and investigated generalized semi-open sets, semi generalized open sets, generalized open sets, semi-open sets, pre-open sets and $\alpha$-open sets, semi pre-open sets and $b$-open sets which are some of the weak forms of open sets and the complements of these sets are called the same types of closed sets. 

 In 2014, \cite{milb2014} we introduced the concepts of $\alpha^{m}$-closed sets and $\alpha^{m}$-open set in topological spaces. Also we have introduced the concepts of $\alpha^{m}$-continuous functiopns and $\alpha^{m}$-irresolute functions. In this paper, based on the $\alpha^{m}$-open and $\alpha^{m}$-closed sets we introduce separation axioms of $\alpha^{m}$-open set called $T_{\alpha^{m}}$-space. Further various characterisation are studied. 
are introduced
\section{Preliminaries}
Throughout this paper $(X, \tau)$ and $(Y, \sigma)$ represent topological spaces on which no separation axioms are assumed unless otherwise mentioned. For a subset $A$ of a space $(X, \tau)$, $cl(A)$, $int(A)$ and $A^{c}$ denote the closure of $A$, the interior of $A$ and the complement of $A$ in $X$, respectively. \\

\begin{definition}
A subset $A$ of a topological space $(X, \tau)$ is called \\ \\
$(a)$ a preopen set \cite{ma1982} if $A\subseteq int(cl(A))$ and pre-closed set if $cl(int(A))\subseteq A$. \\ \\
$(b)$ a semiopen set \cite{le1963} if $A \subseteq cl(int(A))$ and semi closed set if $int(cl(A))\subseteq A$. \\ \\
$(c)$ an $\alpha$-open set \cite{nj1965} if $A \subseteq int(cl(int(A)))$ and an $\alpha$-closed set if $cl(int(cl(A)))\subseteq A$. \\ \\
$(d)$ a semi-preopen set \cite{an1986} ($\beta$-open set) if $A \subseteq cl(int(cl(A)))$ and semi-preclosed set if $int(cl(int(A))) \subseteq A$. \\ \\
$(e)$ an $\alpha^{m}$-closed set \cite{milb2014} if $int(cl(A))\subseteq U$, whenever $A\subseteq U$ and $U$ is $\alpha$-open. The complement of $\alpha^{m}$-closed set is called an $\alpha^{m}$-open set. \\
\end{definition}

\begin{definition}
A space $(X,\ \tau_{X})$ is called a $T_{\frac{1}{2}}$-space \cite{le1970} if every $g$-closed set is closed. \\
\end{definition}

\begin{definition} 
A function $f:(X,\ \tau)\longrightarrow(Y,\ \sigma)$ is called \\ \\
$(a)$ an $\alpha^{m}$-continuous \cite{milb} if $f^{-1}(V)$ is $\alpha^{m}$-closed in $(X,\ \tau)$ for every closed set $V$ of $(Y,\ \sigma)$. \\ \\
$(b)$ an $\alpha^{m}$-irresolute \cite{milb} if $f^{-1}(V)$ is $\alpha^{m}$-closed in $(X,\ \tau)$ for every $\alpha^{m}$-closed set $V$ of $(Y,\ \sigma)$. \\
\end{definition}

\section{Separation axioms of $\alpha^{m}$-open sets}
\begin{definition}
A space $(X, \tau)$ is called $T_{\alpha^{m}}$-space if every $\alpha^{m}$-closed set is closed. \\
\end{definition}

\begin{theorem}
For a topological space $(X, \tau)$ the following conditions are equivalent. \\ \\
$(a)$ $(X, \tau)$ is a $T_{\alpha^{m}}$-space. \\ \\
$(b)$ Every singleton $\left\{x\right\}$ is either $\alpha$-closed (or) clopen. \\
\end{theorem}
\begin{proof}
$(a)\Rightarrow (b)$ Let $x\in X$. Suppose $\left\{x\right\}$ is not an $\alpha$-closed set of $(X, \tau)$. Then $X-\left\{x\right\}$ is not an $\alpha$-open set. Thus $X-\left\{x\right\}$ is an $\alpha$-closed set of $(X, \tau)$. Since $(X, \tau)$ is a $T_{\alpha^{m}}$-space, $X-\left\{x\right\}$ is a $\alpha$-closed set of $(X, \tau)$, i.e., $\left\{x\right\}$ is $\alpha$-open set of $(X, \tau)$. \\ \\
$(b)\Rightarrow (a)$ Let $A$ be an $\alpha^{m}$-closed set of $(X, \tau)$. Let $x\in int(cl(A))$ by $(b)$, $\left\{x\right\}$ is either $\alpha$-closed (or) clopen. \\ \\
$Case (i)$: Let $\left\{x\right\}$ be an $\alpha$-closed. If we assume that $x\notin A$, then we would have $x\in int(cl(A))-A$ which cannot happen. Hence $x\in A$.  \\ \\
$Case (ii)$: Let $\left\{x\right\}$ be a clopen. Since $x\in int(cl(A))$, then $\left\{x\right\}\bigcap A\neq\phi$. This shows that $x\in A$. So in both cases we have $int(cl(A))\subseteq A$. Trivially $A\subseteq int(cl(A))$. Therefore $A=int(cl(A))$ (or) equivalently $A$ is clopen. Hence $(X, \tau)$ is a $T_{\alpha^{m}}$-space. \\
\end{proof}

\begin{proposition}
If $f:(X, \tau)\longrightarrow(Y, \sigma)$ be an $\alpha^{m}$-continuous function and $(X, \tau)$ be a $T_{\alpha^{m}}$-space. Then $f$ is continuous.\\
\end{proposition}
\begin{proof}
Let $V$ be closed set in $(Y, \sigma)$. Since $f$ is an $\alpha^{m}$-continuous function, $f^{-1}(V)$ is an $\alpha^{m}$-closed set in $(X, \tau)$. Since $(X, \tau)$ is a $T_{\alpha^{m}}$-space, $f^{-1}(V)$ is closed set in $(X, \tau)$. Hence $f$ is continuous. \\
\end{proof}

\begin{theorem}
Let $f:(X, \tau)\longrightarrow(Y, \sigma)$ be a mapping and $(X, \tau)$ be a $T_{\alpha^{m}}$-space, then $f$ is continuous if one of the following conditions is satisfied. \\ \\
$(a)$ $f$ is $\alpha^{m}$-continuous. \\ \\
$(b)$ $f$ is $\alpha^{m}$-irresolute. 
\end{theorem}
\begin{proof} Obvious.
\end{proof}

\begin{theorem}
A map $f:(X, \tau)\longrightarrow(Y, \sigma)$ is an $\alpha^{m}$-continuous function if and only if the inverse image of every open set in $(Y, \sigma)$ is an $\alpha^{m}$-open set in $(X, \tau)$. \\
\end{theorem}
\begin{proof} $Necessity:$ Let $f:(X, \tau)\longrightarrow(Y, \sigma)$ be an $\alpha^{m}$-continuous function and $U$ be an open set in $(Y, \sigma)$. Then $Y-U$ is closed in $(Y, \sigma)$. Since $f$ is an $\alpha^{m}$-continuous function, $f^{-1}(Y-U)=X-f^{-1}(U)$ is an $\alpha^{m}$-closed in $(X, \tau)$ and hence $f^{-1}(U)$ is an $\alpha^{m}$-open in $(X, \tau)$. \\ \\
$Sufficiency:$ Assume that $f^{-1}(V)$ is an $\alpha^{m}$-open set in $(X, \tau)$ for each open set $V$ in $(Y, \sigma)$. Let $V$ be a closed set in $(Y, \sigma)$. Then $Y-V$ is an open set in $(Y, \sigma)$. By assumption, $f^{-1}(Y-V)=X-f^{-1}(V)$ is an $\alpha^{m}$-open in $(X, \tau)$, which implies that $f^{-1}(V)$ is an $\alpha^{m}$-closed set in $(X, \tau)$. Hence $f$ is an $\alpha^{m}$-continuous function. \\
\end{proof}

\begin{proposition}
Let $f:(X, \tau)\longrightarrow(Y, \sigma)$ be any topological space and $(Y, \sigma)$ be a $T_{\alpha^{m}}$-space. If $f:(X, \tau)\longrightarrow(Y, \sigma)$ and $g:(Y, \sigma)\longrightarrow(Z, \eta)$ are $\alpha^{m}$-continuous functions, then their composition $g\circ f:(X, \tau)\longrightarrow(Z, \eta)$ is an $\alpha^{m}$-continuous function.\\
\end{proposition} 
\begin{proof}  Let $V$ be a closed set in $(Z, \eta)$. Since $g:(Y, \sigma)\longrightarrow(Z, \eta)$ is an $\alpha^{m}$-continuous function, $g^{-1}(V)$ is an $\alpha^{m}$-closed set in $(Y, \sigma)$. Since $(Y, \sigma)$ is a $T_{\alpha^{m}}$-space, $g^{-1}(V)$ is a closed set in $(Y, \sigma)$. Since $f:(X, \tau)\longrightarrow(Y, \sigma)$ is an $\alpha^{m}$-continuous function, $f^{-1}(g^{-1}(V))=(g\circ f)^{-1}(V)$ is an $\alpha^{m}$-closed set in $(X, \tau)$. Hence $g\circ f:(X, \tau)\longrightarrow(Z, \eta)$ is an $\alpha^{m}$-continuous function. \\
\end{proof} 

\begin{definition}
A map $f:(X,\ \tau)\longrightarrow(Y,\ \sigma)$ is said to be \\ \\
$(a)$ $\alpha^{m}$-closed map if $f(V)$ is $\alpha^{m}$-closed in $(Y,\ \sigma)$ for every closed set $V$ of $(X,\ \tau)$. \\ \\
$(b)$ $\alpha^{m}$-open map if $f(V)$ is $\alpha^{m}$-open in $(Y,\ \sigma)$ for every open set $V$ of $(X,\ \tau)$. \\ 
\end{definition}

\begin{theorem}
Let $f:(X, \tau)\longrightarrow(Y, \sigma)$ and $g:(Y, \sigma)\longrightarrow(Z, \eta)$ be two mappings and $(Y, \sigma)$ be a $T_{\alpha^{m}}$-space, then \\ \\
$(a)$ $g\circ f$ is $\alpha^{m}$-continuous, if $f$ and $g$ are $\alpha^{m}$-continuous. \\ \\
$(b)$ $g\circ f$ is $\alpha^{m}$-closed, if $f$ and $g$ are $\alpha^{m}$-closed. \\
\end{theorem}
\begin{proof}
$(a)$ Let $V$ be a closed set of $(Z, \eta)$, then $g^{-1}(V)$ is $\alpha^{m}$-closed set in $(Y, \sigma)$. Since $(Y, \sigma)$ is a $T_{\alpha^{m}}$-space, then $g^{-1}(V)$ is a closed set in $(Y, \sigma)$. But $f$ is $\alpha^{m}$-continuous, then $(g\circ f)^{-1}(V)=f^{-1}(g^{-1}(V))$ is $\alpha^{m}$-closed in $(X, \tau)$ this implies that $g\circ f$ is $\alpha^{m}$-continuous mappings. \\ \\
$(b)$ The proof is similar. \\
\end{proof}

\begin{theorem}
Let $f:(X, \tau)\longrightarrow(Y, \sigma)$ be a mapping from a $T_{\alpha^{m}}$-space $(X, \tau)$ into a space $(Y, \sigma)$, then \\ \\
$(a)$ $f$ is continuous mapping if, $f$ is $\alpha^{m}$-continuous. \\ \\
$(b)$ $f$ is closed mapping if, $f$ is $\alpha^{m}$-closed. \\ 
\end{theorem}
\begin{proof} Obvious. \\
\end{proof}

\begin{theorem}
Let $f:(X, \tau)\longrightarrow(Y, \sigma)$ is surjective closed and $\alpha^{m}$-irresolute, then $(Y, \sigma)$ $T_{\alpha^{m}}$-space if $(X, \tau)$ is also $T_{\alpha^{m}}$-space. \\
\end{theorem}
\begin{proof}
Let $V$ be an $\alpha^{m}$-closed subset of $(Y, \sigma)$. Then $f^{-1}(V)$ is $\alpha^{m}$-closed set in $(X, \tau)$. Since, $(X, \tau)$ is a $T_{\alpha^{m}}$-space, then $f^{-1}(V)$ is closed set in $(X, \tau)$. Hence, $V$ is closed set in $(Y, \sigma)$ and so, $(Y, \sigma)$ is $T_{\alpha^{m}}$-space. \\
\end{proof}

\begin{proposition}
If $f:(X,\ \tau)\longrightarrow(Y,\ \sigma)$ is $\alpha^{m}$-closed, $g:(Y, \sigma)\longrightarrow(Z, \eta)$ is an $\alpha^{m}$-closed, and $(Y,\ \sigma)$ is a $T_{\alpha^{m}}$-space, then their composition $g\circ f:(X, \tau)\longrightarrow(Z, \eta)$ is $\alpha^{m}$-closed. \\
\end{proposition}
\begin{proof}
Let $A$ be a closed set of $(X,\ \tau)$. Then by assumption $f(A)$ is $\alpha^{m}$-closed in $(Y,\ \sigma)$. Since $(Y,\ \sigma)$ is a $T_{\alpha^{m}}$-space, $f(A)$ is closed in $(Y,\ \sigma)$ and again by assumption $g(f(A))$ is $\alpha^{m}$-closed in $(Z, \eta)$. i.e., $(g\circ f)(A)$ is $\alpha^{m}$-closed in $(Z, \eta)$ and so $g\circ f$ is $\alpha^{m}$-closed. \\ 
\end{proof}

\begin{proposition}
For any bijection $f:(X, \tau)\longrightarrow(Y, \sigma)$ the following statements are equivalent: \\ \\
$(a)$ $f^{-1}:(X, \tau)\longrightarrow(Y, \sigma)$ is $\alpha^{m}$-continuous. \\ \\
$(b)$ $f$ is $\alpha^{m}$-open map. \\ \\
$(c)$ $f$ is $\alpha^{m}$-closed map. \\
\end{proposition}
\begin{proof}
$(a)\Longrightarrow (b)$ Let $U$ be an open set of $(X, \tau)$. By assumption, $(f^{-1})^{-1}(U)= f(U)$ is $\alpha^{m}$-open in $(Y, \sigma)$ and so $f$ is $\alpha^{m}$-open. \\ \\
$(b)\Longrightarrow (c)$ Let $F$ be a closed set of $(X, \tau)$. Then $F^{c}$ is open set in $(X, \tau)$. By assumption, $f(F^{c})$ is $\alpha^{m}$-open in $(Y, \sigma)$. That is $f(F^{c}) = (f(F))^{c}$ is $\alpha^{m}$-open in $(Y, \sigma)$ and therefore $f(F)$ is $\alpha^{m}$-closed in $(Y, \sigma)$. Hence $f$ is $\alpha^{m}$-closed. \\ \\
$(c)\Longrightarrow (a)$ Let $F$ be a closed set of $(X, \tau)$. By assumption, $f(F)$ is $\alpha^{m}$-closed in $(Y, \sigma)$. But $f(F)=(f^{-1})^{-1}(F)$ and therefore $f^{-1}$ is $\alpha^{m}$-continuous. \\
\end{proof}



\begin{thebibliography}{15}\label{bibliography}

\bibitem{an1986} Andrijevic. D, Semi-preopen sets, Mat. Vesink, 38 (1986), 24 - 32. \\

\bibitem{an1996} Andrijevic. D, On $b$-open sets,Mat. Vesink, 48 (1996), 59 - 64.  \\

\bibitem{ar1990} S. P. Arya and T. Nour, Characterizations of $s$-normal spaces, Indian J. Pure Appl. Math., 21 (1990), 717-719.  \\

\bibitem{bh1987} Bhattacharyya. P and Lahiri. B. K, Semi-generalizsed closed sets in topology, Indian J. Math., 29 (1987), 375 - 382.  \\

\bibitem{khal1970} E.D. Khalimsky, Applications of connected ordered topological spaces in topology, Conference of Math. Departments of Povolsia (1970).  \\

\bibitem{khal1990} E. Khalimsky, R. Kopperman, P.R. Meyer, Computer graphics and connected topologies on finite ordered sets , Topology Appl., 36, No. 1, (1990), 1-17.  \\

\bibitem{kong1991} T.Y. Kong, R. Kopperman, P.R. Meyer, A topological approach to digital topology, Amer. Math. Monthly, 98 (1991), 901-917.  \\

\bibitem{le1963} Levine. N, Semi-open sets and semi-continuity in topological spaces, Amer. Math. Monthly, 70 (1963), 36 - 41.  \\

\bibitem{le1970} Levine. N, Generalized closed sets in topology, Rend. Circ. Math. Palermo, 19 (1970), 89 - 96.  \\

\bibitem{ma1982}Mashhour. A. S, Abd EI-Monsef. M. E. and EI-Deeb. S. N., On Precontinuous and weak pre-continuous mapping, Proc. Math., Phys. Soc. Egypt, 53 (1982), 47 - 53.  \\

\bibitem{milb2014} Milby Mathew and R. Parimelazhagan, $\alpha^{m}$-Closed Sets in Topological Spaces, International Journal of Mathematical Analysis
Vol. 8, 2014, no. 47, 2325 - 2329. \\

\bibitem{milb} Milby Mathew and R. Parimelazhagan, $\alpha^{m}$-Continuous Maps in Topological Spaces. (submitted)  \\

\bibitem{nj1965} Njastad. O, On some classes of nearly open sets, Pacific J. Math., , 15 (1965), 961-970.  \\


\end{thebibliography}
\end{document}